\documentclass[12pt, reqno]{amsart}
\usepackage{amsmath, amsthm, amscd, amsfonts, amssymb, graphicx, color}

\usepackage{tikz}
\usepackage{subcaption}

\newtheorem{theorem}{Theorem}[section]
\newtheorem{lemma}[theorem]{Lemma}
\newtheorem{proposition}[theorem]{Proposition}
\newtheorem{corollary}[theorem]{Corollary}

\theoremstyle{definition}
\newtheorem{definition}[theorem]{Definition}
\newtheorem{example}[theorem]{Example}

\theoremstyle{remark}

\numberwithin{equation}{section}

\newcommand{\diam}{{\rm diam}\,}
\newcommand{\Ext}{{\rm Ext}\,}
\newcommand{\sgn}{{\rm sgn}\,}
\newcommand{\conv}{{\rm conv}\,}

\begin{document}

\title[Approximate smoothness]{Approximate smoothness\\ 
in normed linear spaces}

\author[J. Chmieli\'{n}ski, D. Khurana and D. Sain]{Jacek Chmieli\'{n}ski, Divya Khurana and Debmalya Sain}

\address[Chmieli\'{n}ski]{Department of Mathematics\\ Pedagogical
University of Krakow\\ Podchor\c{a}\.{z}ych 2, 30-084 Krak\'{o}w\\
Poland} \email{jacek.chmielinski@up.krakow.pl}

\address[Khurana]{Humanities and Applied Sciences\\ IIM Ranchi\\Suchana Bha\-wan, Audrey House Campus, Meur’s Road, Ranchi\\ Jharkhand-834008 \\India}
\email{divyakhurana11@gmail.com}

\address[Sain]{Departamento de Analisis Matematico\\ Facultad de Ciencas\\Universidad de Granada\\ Avenida de la Fuente Nueva\\ S/N, C.P: 18071, Granada\\Spain} 
\email{saindebmalya@gmail.com}

\thanks{}

%\date{November 1, 2022}



\begin{abstract}
We introduce the notion of approximate smoothness in a normed linear space. We characterize this property and show the connections between smoothness and approximate smoothness for some spaces. As an application, we consider in particular the Birkhoff-James orthogonality and its right-additivity under the assumption of approximate smoothness. 
\end{abstract}

\maketitle

\section{Introduction}

Smoothness is definitely one of the most important geometrical properties of normed linear spaces (cf. monographs \cite{Day,JL,Megginson} or a survey \cite{FMZ} for example). In particular, some natural attributes of the Birkhoff-James orthogonality relation can be derived for smooth spaces. Since in its full strength, smoothness can be sometimes a too much restrictive assumption, we are going to propose somehow relaxed, approximate version of this property. We study this concept in general, as well as for some particular linear normed spaces,  polyhedral Banach spaces and direct sums of normed linear spaces.

\subsection{Notations}
Throughout the text, we use the symbols
$X,Y,Z$ to denote real normed linear spaces. The zero vector of a normed linear space is denoted by $\theta$, but in case of the scalar field $\mathbb{R}$, we simply use the symbol 0. By $B_X:=\{x\in X: \|x\|\leq1\}$ and $S_X:=\{x\in X: \|x\|=1\}$ we denote the unit ball and the unit sphere of $X$, respectively. The collection of all extreme points of $B_X$ will be denoted as $\Ext B_X$. For $A\subset X$, $\diam A:=\sup\limits_{x,y\in A}\|x-y\|$ denotes the diameter of $A$.

Let $X^*$ denote the dual space of $X$. Given $x\in X\smallsetminus\{\theta\}$, a functional $f\in S_{X^*}$ is said to be a {\em supporting functional at} $x$ if $f(x)=\|x\|$. The collection of all supporting functionals at $x$ will be denoted by $J(x)$, i.e.,
$$
J(x):=\{f\in S_{X^{*}}:\ f(x)=\|x\|\},\qquad x\in X\smallsetminus\{\theta\}.
$$
The Hahn--Banach theorem guarantees that the set $J(x)$ is always nonempty and it is easy to see that it is also convex.  It is also known that $J(x)$ is $w^*$-compact. An element $x\in X\smallsetminus\{\theta\}$ is said to be a {\em smooth point} if $J(x)$ is a singleton (i.e., $J(x)=\{f\}$ for a unique $f\in S_{X^*}$). A normed linear space $X$ is called {\it smooth} if every $x\in S_X$ (hence every $x\in X\smallsetminus\{\theta\}$) is a smooth point.

For $f\in X^*\smallsetminus\{\theta\}$, we write 
$$
M_{f}:=\{x\in S_{X}: |f(x)|=\|f\|\}
$$ 
and 
$$
M_{f}^{+}:=\{x\in S_{X}: f(x)=\|f\|\}.
$$
Given two elements $x,y\in X$, let 
$$
\overline{xy}{:=}\mbox{conv}\,\{x,y\}=\{(1-t)x+ty:t\in[0,1]\}
$$ 
denote the closed line segment joining $x$ and $y$.
By $\mathcal{R}(X)$ we will denote the length of the ``longest'' line segment lying on a unit sphere (cf. \cite{SW,CKS}); more precisely, 
$$
\mathcal{R}(X):=\sup\{\|x-y\|:\overline{xy}\subset S_X\}.
$$
By a hyperplane we mean a set
$$
H{:=}\{x\in X:\ f(x)=c\}
$$
where $f\in X^*\smallsetminus\{\theta\}$ is a functional and $c\in\mathbb{R}$ a constant. Each hyperplane $H\subset X$ divides $X$ into two closed half-spaces whose intersection is $H$ itself. We call $H$ a \emph{supporting hyperplane to the unit ball}, if $B_X$ lies entirely within one of the two half-spaces and $H\cap B_X\neq\emptyset$. Equivalently, $H$ is a supporting hyperplane for $B_X$ if and only if there exists $f\in S_{X^*}$ such that 
$H=\{x\in X:\ f(x)=1\}$. Notice that then we have $H\cap S_X=M_f^+$.

By $\mathfrak{H}(X)$ we denote the set of all supporting hyperplanes for the unit ball in $X$. Now, we introduce the notion  
$$
\mathcal{S}(X){:=}\sup\{\diam (H\cap S_X):\ H\in \mathfrak{H}(X)\}
$$
and by previous observations we have
\begin{equation}\label{S(X)}
\mathcal{S}(X)=\sup\{\diam M_f^+:\ f\in S_{X^*}\}.
\end{equation}

Actually, by using the fact that the diameter a of convex set is the supremum of lengths of segments inside the set and both $H\cap S_X$ (for $H\in \mathfrak{H}(X)$) and $M^{+}_{f}$ are convex sets, it follows that \begin{equation}\label{R(X)}
\mathcal{S}(X)=\mathcal{R}(X).
\end{equation}
 We define $d:X\smallsetminus \theta \longrightarrow [0,2]$ by $d(x)=\diam J(x).$
We also set 
$$\mathcal{E}(X){:=}\sup\{{d(x)}:\ x\in S_X\}.
$$

By $\Psi: X\longrightarrow X^{**}$  we denote the  canonical embedding of a normed linear space $X$ into its bidual $X^{**}$, that is
$$
\Psi(x)(f){:=}f(x),\qquad x\in X, f\in X^*.
$$
It is known that $\Psi$ is a linear isometry and if it is surjective ($\Psi(X)=X^{**}$), then $X$ is called a {\em reflexive} space (necessarily a Banach space). 

By using the observation that $\Psi(M_{f}^{+})=J(f)\cap \Psi(X)$ and the fact that $\Psi$ is an isometry we notice that, for any normed linear space $X$:
\begin{equation}\label{relation between e and s *}
\mathcal{E}(X)\leq \mathcal{S}(X^*) \quad \mbox{and}\quad  \mathcal{S}(X)\leq \mathcal{E}(X^*).
\end{equation}

In particular we have:

\begin{equation}\label{relation between e and s}
\mathcal{E}(X)\leq \mathcal{S}(X^*)\leq \mathcal{E}(X^{**}).
\end{equation}

An $n$-dimensional Banach space $X$ is said to be a {\it polyhedral Banach space} if  $B_X$ contains only finitely many extreme points or, equivalently, if $S_X$ is a polyhedron. For more details on polyhedral Banach spaces see \cite{SPBB}.   

We define the sign function on $\mathbb{R}$ by
$$
\sgn t:=\frac{t}{|t|},\quad t\in\mathbb{R}\smallsetminus\{0\}\quad \mbox{and}\quad \sgn 0:=0.
$$ 
Let $(X,\|\cdot\|_X)$ and $(Y, \|\cdot\|_Y)$ be normed linear spaces. Then 
$$
X\oplus_p Y:=\{(x,y):x\in X, y\in Y\},\qquad 1\leq p \leq \infty,
$$ 
is a normed linear space with respect to the following norm:
$$
\|(x,y)\|_p :=
\begin{cases}
(\|x\|_X^p+\|y\|_Y^p)^{\frac{1}{p}}& \text{if $p<\infty$} \\
\max\{\|x\|_X,\|y\|_Y\} & \text{if $p=\infty$.} 
\end{cases}
$$

Let $1\leq q\leq\infty$ be conjugated to $p$, i.e., such that $\frac{1}{p}+\frac{1}{q}=1$ for $1<p<\infty$, $q=1$ for $p=\infty$ and $q=\infty$ for $p=1$. Then the dual space of $X\oplus_p Y$ can be isometrically identified with $X^*\oplus_q Y^*$ in the following sense:
for each $F\in (X\oplus_p Y)^*$ there exist a unique $(f,g)\in X^*\oplus_q Y^*$ such that $\|F\|=\|(f,g)\|$ and $F((x,y))=f(x)+g(y)$ for all $(x,y)\in X\oplus_p Y$ (compare \cite[p. 5]{JL} or \cite[Definition 0.18 and Proposition 0.19]{Dulst}).

\subsection{Norm derivatives}

It is well known that smoothness of $x\in S_X$ is
equivalent to the G\^{a}teaux differentiability of the norm at $x$. The concept of norm derivatives arises naturally from the two-sided limiting nature of the G\^{a}teaux
derivative of the norm. Let us recall the following definition and basic properties of one-sided norm derivatives.

\begin{definition}\label{def_norm_der}
Let $X$ be a normed linear space and $x,y\in X$. The norm derivatives of $x$ in the direction of $y$ are defined as
\begin{align*}
\rho'_{+}(x,y):=\|x\|\lim_{\lambda\rightarrow 0^{+}}\frac{\|x+\lambda
y\|-\|x\|}{\lambda}=\lim_{\lambda\rightarrow 0^{+}}\frac{\|x+\lambda y\|^2-\|x\|^2}{2\lambda},\\
\rho'_{-}(x,y):=|x\|\lim_{\lambda\rightarrow
0^{-}}\frac{\|x+\lambda y\|-\|x\|}{\lambda}=\lim_{\lambda\rightarrow 0^{-}}\frac{\|x+\lambda y\|^2-\|x\|^2}{2\lambda}.
\end{align*}
\end{definition}

The following properties of norm derivatives will be useful in this note (see \cite{AST} and \cite{D1} for proofs):

\begin{itemize}
\item[($\rho$-i)]For all $x$, $y\in X$ and all $\alpha\in\mathbb{R}$,
\[\rho'_\pm(\alpha x, y)=\rho'_\pm(x,\alpha y) = \left\{
    \begin{array}{ll}
        \alpha \rho'_\pm(x,y)  & \mbox{if } \alpha\geq 0 \\
       \alpha \rho'_\mp(x,y)  & \mbox{if } \alpha< 0.
    \end{array}
\right.\]

\item[($\rho$-ii)] $\rho'_{-}(x,y)\leq \rho'_{+}(x,y)$. Moreover, $x\in X\smallsetminus\{\theta\}$ is smooth if and only if $\rho'_{-}(x,y)= \rho'_{+}(x,y)$ for all $y\in X$.

\item[($\rho$-iii)] $\rho'_{+}(x,y)=\|x\|\sup\{f(y):f\in J(x)\}$.

\item[($\rho$-iv)] $\rho'_{-}(x,y)=\|x\|\inf\{f(y):f\in J(x)\}$.
\end{itemize}

\subsection{Birkhoff-James orthogonality}

For vectors $x$ and $y$ in a nor\-med space $X$, we say that $x$ is {\it Birkhoff-Ja\-mes orthogonal} (BJ-ortho\-go\-nal) to $y$, written as $x\perp_B y$, if 
$$
\|x+\lambda y\|\geq \|x\|,\qquad \mbox{for all}\ \lambda\in \mathbb{R}
$$
(cf. \cite{B, J1, J}). James in \cite[Theorem 2.1]{J1} proved that if $x \in X\smallsetminus\{\theta\}$ and $y\in X$, then $x\perp_B y$ if and only if there exists $f\in J(x)$ such that $f(y)=0$.

Chmieli\'{n}ski in \cite{C} defined an approximate Birkhoff-James orthogonality
as follows. Given $x,y \in X$ and $\varepsilon\in
[0,1)$, $x$ is said to be approximately orthogonal to $y$, written
as $x\perp_B^\varepsilon y$, if 
\begin{equation}\label{approx_ort}
\|x+\lambda y\|^2\geq \|x\|^2-2\varepsilon\|x\|\|\lambda y\|,\qquad \mbox{for all}\ \lambda \in
\mathbb{R}.
\end{equation} 
Later, in \cite[Theorems 2.2 and 2.3]{CSW}, Chmieli\'{n}ski et al.
gave two characterizations of this approximate orthogonality:
\begin{eqnarray}
x\perp_B^\varepsilon y &\Longleftrightarrow& \exists\,z\in \mbox{span}\{x,y\}:\  x\perp_B z,\ \|z-y\|\leq\varepsilon \|y\|;\\ \label{Approximate_Characterization}
& \Longleftrightarrow& \exists\,f\in
J(x):\ |f(y)|\leq \varepsilon \|y\|. \label{Bchar}
\end{eqnarray}

Obviously for any $f$ in $J(x)$ we have $|f(y)|\leq \|y\|$. If $|f(y)|<\|y\|$ then there exists $\varepsilon\in [0,1)$ such that $|f(y)|=\varepsilon\|y\|$ whence $x\perp_B^{\varepsilon}y$. In particular, it follows from \eqref{Bchar} for $x,y\in S_X$:
\begin{equation}\label{Bchar2}
x\perp_B^\varepsilon y\ \mbox{for some}\ \varepsilon\in [0,1)  \ \Longleftrightarrow \ y\not\in M_f\ \mbox{for some}\ f\in J(x)
\end{equation}
or equivalently
\begin{equation}\label{Bchar3}
x\not\perp_B^\varepsilon y\ \mbox{for all}\ \varepsilon\in [0,1)  \ \Longleftrightarrow \ y\in M_f\ \mbox{for all}\ f\in J(x).
\end{equation}

The study of approximate orthogonality has been an active area of research in recent times, particularly in the space of bounded linear operators on a Banach space (see \cite{CSW,MPRS,KDM1, SKM} for recent study on approximate orthogonality). Since Birkhoff-James orthogonality is closely related to the classical notion of smoothness in a normed space, the above works motivated us to introduce a suitable notion of approximate smoothness in normed spaces which will be in some sense compatible with approximate orthogonality.

In \cite{J1}, James obtained a characterization of smooth points in terms of right-additivity of the Birkhoff-James orthogonality relation. Na\-me\-ly, $x\in X\smallsetminus\{\theta\}$ is a smooth point in $X$ if and only if $\perp_B$ is right-additive at $x$, that is for any $y,z \in X$:
$$
x\perp_B y\ \ \mbox{and}\ \ x\perp_B z\ \implies\ x\perp_B (y+z).
$$ 

We now define the notion of an approximate right-additivity of the BJ-orthogonality.

\begin{definition}\label{d1}
Let $X$ be a normed linear space and $x\in X\smallsetminus\{\theta\}$. 
We say that the BJ-orthogonality is \textit{approximately right-additive} at $x$ ($\varepsilon$-right-additive for some $\varepsilon\in[0,1)$) if 
for any $y,z \in X$:
$$
x\perp_B y\ \ \mbox{and}\ \ x\perp_B z\ \implies\ x\perp_B^{\varepsilon} (y+z).
$$
\end{definition}

Similarly we define the right-additivity of the approximate BJ-or\-tho\-go\-na\-li\-ty.

\begin{definition}\label{d2}
We say that the approximate BJ-orthogonality is \textit{right-additive} at $x$ if whenever $x\perp_B^\varepsilon y$, $x\perp_B^\varepsilon z$ for some $y,z \in X$ and $\varepsilon\in[0,1)$, then there exists some $\varepsilon_1\in[0,1)$ such that $x\perp_B^{\varepsilon_1} (y+z)$.
\end{definition}

Observe that in Definition~\ref{d1} for a given $x\in X\smallsetminus\{\theta\}$, $\varepsilon$ is uniform in the sense that it is independent on the choice of $y$ and $z$ and in the Definition~\ref{d2} for given $x,y,z$ and $\varepsilon$, $\varepsilon_1$ may depend on all of them. 

The connections between smoothness and right-additivity of the BJ-orthogonality induce a natural question on a condition, weaker than smoothness, which would characterize the approximate right-additivity of the BJ-orthogonality (or right-additivity of the approximate BJ-orthogonality). This is one of the motivations that lead to the notion of approximate smoothness which is stated in the next section.

\section{Approximate smoothness and rotundity}

\subsection{Approximate smoothness -- definition and basic properties}
It is obvious that for any $x\in X\smallsetminus\{\theta\}$, $0\leq {d(x)}\leq 2$ and $d(x)=0$ if and only if $x$ is a smooth point. We will be considering the case when the set $J(x)$ is not necessarily a singleton, but its diameter is limited, in particular strictly less than $2$. 
\begin{definition}\label{d3}
Let $X$ be a normed linear space, $x\in X\smallsetminus\{\theta\}$ and $\varepsilon\in [0,2)$. We say that $x$ is {\em $\varepsilon$-smooth} if ${d(x)}\leq\varepsilon$. When the value of $\varepsilon$ is not specified we say that $x$ is {\em approximately smooth}.
The space $X$ is said to be approximately smooth ($\varepsilon$-smooth) if each $x\in S_X$ is $\varepsilon_x$-smooth for some $\varepsilon_x\leq \varepsilon<{2}$. 
\end{definition}

Observe that in the above definition we intentionally excluded the possibility of $\varepsilon=2$. Otherwise, every non-zero element of a normed space is $\varepsilon$-smooth and every normed space is $\varepsilon$-approximately smooth for some $\varepsilon\in[0,2]$. However, our motivation is to distinguish approximately smooth points and spaces, so we restrict to $\varepsilon<2$.

The following lemma will be useful for obtaining a characterization of approximate smoothness in terms of norm derivatives.

\begin{lemma}\label{d(x) equality}
Let $X$ be a normed linear space, $x\in X\smallsetminus\{\theta\}$. Then  
$$\sup\limits_{y\in S_X}\{\rho'_+(x,y)-\rho'_{-}(x,y)\}=d(x) \|x\|.$$
\end{lemma}

\begin{proof}
Using the properties ($\rho$-iii) and ($\rho$-iv), for an arbitrary $y\in S_X$ we have 
\begin{align*}
 \rho'_+(x,y)-\rho'_{-}(x,y)&=\|x\|\{\sup_{f\in J(x)}f(y)-\inf_{g\in J(x)}g(y)\}\\
 &=\|x\| \sup_{f,g\in J(x)}\{f(y)-g(y)\}.
\end{align*}
Hence $\sup\limits_{y\in S_X}\{\rho'_+(x,y)-\rho'_{-}(x,y)\}= d(x) \|x\|$.
\end{proof}

We  now prove the said characterization.

\begin{lemma}\label{equivalence}
Let $X$ be a normed linear space, $x\in X\smallsetminus\{\theta\}$ and $\varepsilon\in[0,2)$. Then the following conditions are equivalent:
\begin{itemize}
    \item[(i)] $x$ is $\varepsilon$-smooth.
    \item[(ii)] $\sup\limits_{y\in S_X}\{\rho'_+(x,y)-\rho'_{-}(x,y)\}\leq \varepsilon \|x\|$.
\end{itemize}
\end{lemma}

\begin{proof}
(i)$\implies$ (ii) follows from Lemma~\ref{d(x) equality}.

Now we prove (ii)$\implies$(i). Let $f,g\in J(x)$ and $y\in S_X$. Then $\rho'_+(x,y)\geq \|x\| f(y)$, $\rho'_+(x,y)\geq \|x\| g(y)$, $\rho'_-(x,y)\leq \|x\| f(y)$ and  $\rho'_-(x,y)\leq \|x\| g(y)$. Thus,
\begin{align*}
\sup\limits_{y\in S_X}\{\rho'_+(x,y)-\rho'_{-}(x,y)\}\geq \|x\|\sup_{y\in S_X}|f(y)-g(y)|=\|x\|\|f-g\|. 
\end{align*}
Hence, (ii) yields $\|x\|\,\|f-g\|\leq\varepsilon\|x\|$ and this proves (i).
\end{proof}

As an immediate application of Lemma~\ref{equivalence} we obtain a characterization of the approximate smoothness of the whole space.

\begin{corollary}
Let $X$ be a normed linear space and $\varepsilon\in[0,2)$. Then the following conditions are equivalent:
\begin{itemize}
    \item[(i)] $X$ is $\varepsilon$-smooth.
    \item[(ii)] $\sup\limits_{x,y\in S_X}\{\rho'_+(x,y)-\rho'_{-}(x,y)\}\leq \varepsilon$.
\end{itemize}
\end{corollary}

 If $X$ is a reflexive Banach space then we notice that $\Psi(M_{f}^{+})=J(f)$ for $f\in X^*\smallsetminus\{\theta\}$.
Since $\Psi$ is an isometry so it follows that $f\in X^*\smallsetminus\{\theta\}$ is $\varepsilon$-smooth if and only if $\diam M_{f}^{+}\leq \varepsilon$. Also, for a reflexive Banach space $X$ by using (\ref{relation between e and s}) and $\mathcal{E}(X)=\mathcal{E}(X^{**})$ we get 
\begin{align}
X  \mbox{~is~} \varepsilon \mbox{-smooth if and only if~} \mathcal{S}(X^*)\leq\varepsilon. 
\end{align}

Now, we consider a finite-dimensional space. We show that if each vector on a unit sphere is approximately smooth (not necessarily with the same approximation constant), then so is the whole space. 
Basically, this result is a consequence of Weierstrass compactness theorem.

\begin{theorem}\label{local-global-appr-smooth}
Let $X$ be a finite-dimensional Banach space such that each $x\in S_X$ is $\varepsilon_x$-smooth for some $\varepsilon_x\in[0,2)$. Then $X$ is approximately smooth. 
\end{theorem}

\begin{proof}
Let $x\in S_X$. Without loss of generality we assume that $\varepsilon_x=d(x)$. 

Let $\varepsilon:=\sup_{x\in S_X}\varepsilon_x$. Obviously, $\varepsilon\leq 2$ and suppose, contrary to our claim, that $\varepsilon=2$. Then we could find $\{x_n\}\subset S_X$ such that $\varepsilon_{x_n}\nearrow 2$. Also, for each $n\in\mathbb{N}$ we can choose $f_n,g_n\in J(x_n)$ such that $\|f_n-g_n\|>\varepsilon_{x_n}-\frac{1}{n}$.

 Now, using the compactness of $S_X$, $S_{X^*}$ and the fact that $\mathcal{A}=\{(x,f,g):x\in S_X, f,g\in J(x)\}$ is a closed set in $S_{X}\times S_{X^*}\times S_{X^*}$ we can find a convergent subsequence of $\{(x_n,f_n,g_n)\}$ which we again denote by $\{(x_n,f_n,g_n)\}$. Let $(x,f,g)\in\mathcal{A}$ be the limit of $\{(x_n,f_n,g_n)\}$. Let $h:X\times X^*\times X^*\longrightarrow \mathbb{R}$ be defined by $h((x,f,g))=\|f-g\|$. Then $h$ is clearly a continuous function. Thus $\lim_{n \to \infty}h((x_n,f_n,g_n))=h((x,f,g))$ and this implies that $\|f-g\|=2$. This contradicts our assumption that $\varepsilon_x\in[0,2)$, whence $\varepsilon<2$ and $X$ is $\varepsilon$-smooth.

\end{proof}

A normed space is smooth if and only each of its two-dimensional subspaces is smooth (cf. \cite[Proposition 5.4.21]{Megginson}). An analogous result can be proved for approximate smoothness. However, if the approximation constant $\varepsilon$ is not fixed we can prove it merely for finite-dimensional spaces. The authors do not know whether this assumption is essential.

\begin{theorem}\label{two dimensional subspace}
\ 
\begin{itemize}
\item [(i)] Let $X$ be a normed linear space and $\varepsilon\in[0,2)$. Then $X$ is $\varepsilon$-smooth if and only if each of its two-dimensional subspaces is $\varepsilon$-smooth.
\item [(ii)] Let $X$ be a normed linear space. If $X$ is approximately smooth then so is each of its two-dimensional subspaces. Moreover, if dimension of  $X$ is finite and each of its two-dimensio\-nal subspaces is approximately smooth then $X$ is approximately smooth.
\end{itemize}
\end{theorem}
\begin{proof}
Ad (i). Let $X$ be $\varepsilon$-smooth space and $Y$ its two-dimensional subspace. For $y\in Y\smallsetminus\{\theta\}$ any two supporting functionals at $y$, $\tilde{f},\tilde{g}\in S_{Y^*}$ can be extended (by Hahn-Banach theorem) to $f,g\in S_{X^*}$ --- supporting functionals at $y$ in $X$. Now, $\|\tilde{f}-\tilde{g}\|=\|f-g\|\leq\varepsilon$ implies that $Y$ is $\varepsilon$-smooth.

Conversely, let every two-dimensional subspace of $X$ be $\varepsilon$-smooth. Suppose, on the contrary, that $X$ is not $\varepsilon$-smooth. Then there exist $x\in S_X$ and $f,g\in J(x)$ such that $\|f-g\|>\varepsilon$. Let $y\in S_X$ be such that $|(f-g)(y)|>\varepsilon$. Clearly, this choice of $y$ implies that $x$ and $y$ are linearly independent. 
Let $Y=\mbox{span}\{x,y\}$, $\tilde{f}=f|_Y$ and $\tilde{g}=g|_Y$. Then $\tilde{f}(x)=1$, $\tilde{g}(x)=1$ implies that $\tilde{f},\tilde{g}\in S_{Y^*}$. Also, $|(f-g)(y)|>\varepsilon$ implies that $\|\tilde{f}-\tilde{g}\|>\varepsilon$. Thus $\tilde{f},\tilde{g}\in \{h\in S_{Y^*}: h(x)=1\}$ and $\|\tilde{f}-\tilde{g}\|>\varepsilon$. This leads to the contradiction with $\varepsilon$-smoothness of $Y$ and thus $X$ is $\varepsilon$-smooth.

Ad (ii). If $X$ is approximately smooth then, as above, an application of Hahn-Banach theorem yields approximate smoothness of any two-dimensional subspace of $X$.

Now we assume that $X$ is a finite-dimensional space and let every two-dimensional subspace of $X$ be approximately smooth. Suppose, on the contrary, that $X$ is not approximately smooth. Then by Theorem \ref{local-global-appr-smooth} there exist $x\in S_X$ and $f,g\in J(x)$ such that $\|f-g\|=2$. Let $y\in S_X$ be such that $|(f-g)(y)|=2$. Clearly, this choice of $y$ implies that $x$ and $y$ are linearly independent. 
Let $Y=\mbox{span}\{x,y\}$, $\tilde{f}=f|_Y$ and $\tilde{g}=g|_Y$. Now by using arguments similar to the proof of (i) we arrive at the contradiction with approximate smoothness of $Y$ and thus $X$ is approximately smooth.
\end{proof}

The following example shows that in some spaces the notions of smoothness and approximate smoothness can coincide.

\begin{example}
(a) Observe that if $x=(x_1,x_2,x_3,\ldots)\in \ell_1\smallsetminus\{\theta\}$ then $d(x)=2$ if $x_{i_0}=0$ for some $i_0\in\mathbb{N}$ and $d(x)=0$ if $x_i\not=0$ for all $i$. Using this observation it follows that $x\in\ell_1\smallsetminus\{\theta\}$ is $\varepsilon$-smooth for $\varepsilon\in[0,2)$ if and only if $x$ is smooth.

(b) If $x=(x_1,x_2,x_3,\ldots)\in c_0\smallsetminus\{\theta\}$ then $d(x)=2$ if norm of $x$ is attained at more than one coordinate and $d(x)=0$ if norm of $x$ is attained only at one coordinate. Thus it follows that $x\in c_0\smallsetminus\{\theta\}$ is $\varepsilon$-smooth for $\varepsilon\in[0,2)$ if and only if $x$ is smooth.

\end{example}
\subsection{Approximate rotundity} 

Rotundity (strict convexity) is ano\-ther important geometrical property of normed spaces. Although rotundity and smoothness are independent properties, they are related to each other. We would like to show that it is so  with its approximate counterparts.

Let $X$ be a normed linear space and let $\varepsilon\in [0,2)$.
\begin{definition}
We say that $X$ is $\varepsilon$-rotund (or $\varepsilon$-strictly convex) if 
$\mathcal{S}(X)\leq \varepsilon$.
\end{definition}

Obviously, for $\varepsilon=0$ the above condition means that each functional $f\in S_{X^*}$ supports the unit sphere at at most one point, which is equivalent to rotundity (cf. \cite[Theorem 5.1.15, Corollary 5.1.16]{Megginson}).

Based on (\ref{relation between e and s}), we can easily establish connections between approximate smoothness (rotundity) of given space and approximate rotundity (smoothness) of its dual. 

\begin{theorem}\label{duality}
Let $X$ be a normed linear space and let $\varepsilon\in[0,2)$.
\begin{enumerate}
\item
If $X^*$ is $\varepsilon$-smooth, then $X$ is $\varepsilon$-rotund; 
\item
If $X^*$ is $\varepsilon$-rotund, then $X$ is $\varepsilon$-smooth;
\item
If $X$ is reflexive, then $X$ is $\varepsilon$-smooth if and only if $X^*$ is $\varepsilon$-rotund and $X$ is $\varepsilon$-rotund if and only if $X^*$ is $\varepsilon$-smooth.
\end{enumerate}
\end{theorem}

\section{Approximate additivity of the Birkhoff-James orthogonality in approximately smooth spaces}

As we have reminded, the Birkhoff-James orthogonality is locally right-additive at smooth points. Although it is no longer true for non-smooth points, we will show that under a relaxed condition of approximate smoothness, right-additivity can be partially maintained. 

We will start with an example showing that, in general, approximate smoothness does not imply approximate right-additivity of the Birkhoff-James orthogonality.

\begin{example}\label{ex3.1}
Let $X=\mathbb{R}^2$ with a norm generated by
$$
B_{\delta}:=\conv\{(1,1),(0,1+\delta),(-1,1),(-1,-1),(0,-1-\delta),(1,-1)\}
$$
as a unit ball (with $\delta> 0$).

Consider the point $P=(0,1+\delta)$ and functionals 
$$
\begin{array}{c}
f(x,y)=\frac{\delta}{1+\delta}x+\frac{1}{1+\delta} y,\\
g(x,y)=-\frac{\delta}{1+\delta}x+\frac{1}{1+\delta} y.
\end{array}
$$
Notice that $f,g\in J(P)$ (we denote by $L_f$ and $L_g$ the respective supporting lines) and $J(P)=\conv\{f,g\}$ (see Figure 1).
Then
$$
\diam J(P)=\|f-g\|=\sup_{(x,y)\in B_{\delta}}\frac{2\delta}{1+\delta}|x|= \frac{2\delta}{1+\delta}.
$$
Thus the considered space is $\varepsilon$-smooth at $P$, with arbitrarily small $\varepsilon$ (if only $\delta$ is small enough). On the other hand, for $R_1=(1,\delta)$ and $R_2=(-1,\delta)$ we have
$$
P\bot_B R_1,\quad P\bot_B R_2
$$
but $R_1+R_2=(0,2\delta)=\lambda P$ with $\lambda=\frac{2\delta}{1+\delta}$ hence $P\not{\hspace{-0.3em}\bot}_B^{\varepsilon}(R_1+R_2)$ for any $\varepsilon\in [0,1)$.

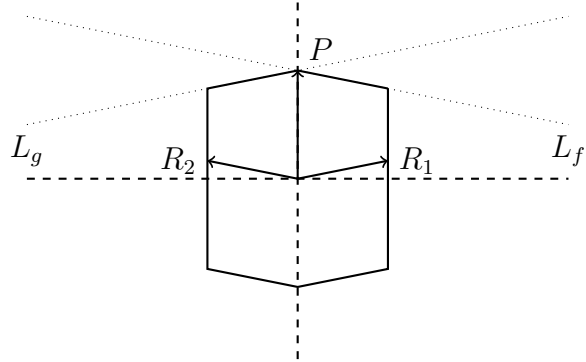
\begin{figure} \label{fig2}
\centering
\begin{tikzpicture}[scale=1.2]

\draw[thick] (1,1)--(0,1.2)--(-1,1)--(-1,-1)--(0,-1.2)--(1,-1)--(1,1);

\draw[thick, black,dashed] (-3,0)--(3,0) node[above]{};
\draw[thick,black,dashed] (0,-2)--(0,2) node[above]{};

\draw (0,1.2)  node[above right] {$P$};
\draw[thick,->] (0,0)--(0,1.2);

\draw (1,0.2)  node[right] {$R_1$};
\draw[thick,->] (0,0)--(1,0.2);

\draw (-1,0.2)  node[left] {$R_2$};
\draw[thick,->] (0,0)--(-1,0.2);

\draw[dotted] (-3,0.6)--(3,1.8);
\draw (-3,0.6)  node[below] {$L_g$};

\draw[dotted] (-3,1.8)--(3,0.6);
\draw (3,0.6)  node[below] {$L_f$};

\end{tikzpicture}
\caption{Illustration to Example \ref{ex3.1}}
\end{figure}
\end{example}

Although approximate smoothness generally does not imply even approximate right-additivity (no matter how small is $\varepsilon$), the following results give some information in particular situations.

\begin{theorem}\label{right-additive}
Let $X$ be a normed linear space and let $x\in X\smallsetminus\{\theta\}$. Let $y_1, y_2\in X\smallsetminus\{\theta\}$ be such that $x\perp_B^\varepsilon y_1$, $x\perp_B^\varepsilon y_2$, where $\varepsilon\in[0,1)$ is such that  
$$
0\leq\varepsilon<\frac{2\|y_1+y_2\|}{3(\|y_1\|+\|y_2\|)}<1.
$$
If $x$ is $\varepsilon$-smooth, then there exists $\varepsilon_1\in[0,1)$ such that $x\perp_B^{\varepsilon_1} (y_1+y_2)$.
\end{theorem}

\begin{proof}
According to \eqref{Bchar} we can find $f,g\in J(x)$ such that $|f(y_1)|\leq \varepsilon\|y_1\|$ and $|g(y_2)|\leq \varepsilon\|y_2\|$. Now, using $\varepsilon$-smoothness of $x$, we get
\begin{align*}
    |f(y_2)|&=|(f-g)(y_2)+g(y_2)|\\
            &\leq \|f-g\|\|y_2\|+\|g(y_2)\|\\
            &\leq \varepsilon\|y_2\|+\varepsilon \|y_2\|\\
            &=2\varepsilon\|y_2\|.
\end{align*}
Using similar arguments we can show that $|g(y_1)|\leq 2 \varepsilon\|y_1\|$. Convexity of $J(x)$ yields  $\frac{1}{2}f+\frac{1}{2}g \in J(x)$ and 
\begin{align*}
    \left|\left(\frac{1}{2}f+\frac{1}{2}g\right)(y_1+y_2)\right|&\leq \frac{1}{2}(\varepsilon \|y_1\|+2\varepsilon\|y_2\|+2\varepsilon\|y_1\|+\varepsilon\|y_2\|)\\
    &=\frac{3}{2}\varepsilon(\|y_1\|+\|y_2\|)\\
    &<\|y_1+y_2\|.
\end{align*}
Let ${\varepsilon_1}\in[0,1)$ be such that $|(\frac{1}{2}f+\frac{1}{2}g)(y_1+y_2)|\leq \varepsilon_1 \|y_1+y_2\| $.  Then \eqref{Bchar} implies that $x\perp_B^{\varepsilon_1} (y_1+y_2)$.
\end{proof}

\begin{theorem}\label{Cor1}
Let $X$ be a normed linear space and let $x\in X\smallsetminus\{\theta\}$. Let $y_1, y_2\in X\smallsetminus\{\theta\}$ be such that $x\perp_By_1$, $x\perp_By_2$. If $x$ is $\varepsilon$-smooth, where $0\leq\varepsilon<\frac{2\|y_1+y_2\|}{\|y_1\|+\|y_2\|}\leq 2$, then there exists ${\varepsilon_1}\in[0,1)$ such that $x\perp_B^{\varepsilon_1} (y_1+y_2)$.
\end{theorem}

\begin{proof}
We find $f,g\in J(x)$ such that $f
(y_1)=0$ and $g(y_2)=0$. Now, by $\varepsilon$-smoothness of $x$, we get
$$
|f(y_2)|=|(f-g)(y_2)|\leq \varepsilon\|y_2\|,
$$
and similarly $|g(y_1)|\leq\varepsilon\|y_1\|$. Clearly, $\frac{1}{2}f+\frac{1}{2}g \in J(x)$ and 
$$
\left |\left(\frac{1}{2}f+\frac{1}{2}g\right)(y_1+y_2)\right|\leq\frac{\varepsilon}{2}(\|y_1\|+\|y_2\|)<\|y_1+y_2\|.
$$
Let ${\varepsilon_1}\in[0,1)$ be such that $\left|\left(\frac{1}{2}f+\frac{1}{2}g\right)(y_1+y_2)\right|\leq \varepsilon_1 \|y_1+y_2\| $.  Thus, \eqref{Bchar} implies that $x\perp_B^{\varepsilon_1} (y_1+y_2)$.
\end{proof}

We now provide an example to show that in  Theorems~\ref{right-additive} and \ref{Cor1}, approximate smoothness of $x$ cannot be omitted. We will need the following lemma.

\begin{lemma}\label{lemma1}
Let $X$ be a normed linear space and let $x\in S_X$. Then the approximate Birkhoff-James orthogonality is not locally right-additive at $x$ if and only if there exist $y_1,y_2\in S_X$, $f_1,f_2\in J(x)$ such that $y_1\not\in M_{f_1}$, $y_2\not\in M_{f_2}$ and $\frac{y_1+y_2}{\|y_1+y_2\|}\in M_f$ for all $f\in J(x)$. 
\end{lemma}

\begin{proof}
The proof relies on characterizations \eqref{Bchar2} and \eqref{Bchar3}. 
Let $x\in S_X$ and suppose that the approximate Birkhoff-James orthogonality is not locally right-additive at $x$, i.e., there exist $y_1,y_2\in S_X$, $\varepsilon\in[0,1)$ such that $x\perp_B^\varepsilon y_1$, $x\perp_B^\varepsilon y_2$ and there does not exist any $\varepsilon_1\in[0,1)$ such that  $x\perp_B^{\varepsilon_1} \frac{y_1+y_2}{\|y_1+y_2\|}$. Equivalently, we can write that (due to \eqref{Bchar2}) there exist $f_1,f_2\in J(x)$ such that $y_1\not\in M_{f_1}$, $y_2\not\in M_{f_2}$ as well as (by \eqref{Bchar3}) that $\frac{y_1+y_2}{\|y_1+y_2\|}\in M_f$ for all $f\in J(x)$. 
\end{proof}

If $x\in X\smallsetminus\{\theta\}$ is $\varepsilon$-smooth and $y_1,y_2\in X$ satisfy the conditions stated in Theorem~\ref{right-additive}, then there exists $f\in J(x)$ such that $\frac{y_1+y_2}{\|y_1+y_2\|}\not\in M_f$. The following example shows that if in Theorems~\ref{right-additive} or \ref{Cor1} the assumption of $x$ being $\varepsilon$-smooth is omitted, the results are not true. 

\begin{example}\label{ex1}
Consider $X=(\mathbb{R}^2,\|~\|_\infty)$. Let $x,y_1,y_2\in S_X$, where $x=(1,1)$, $y_1=(1,\alpha)$, $y_2=(\alpha,1)$, $0<\alpha<\frac{1}{2}$ (see Figure 2(A)). Then, we can observe, using \eqref{Approximate_Characterization}, that
$x\perp_B^\alpha y_1$ and $x\perp_B^\alpha y_2$. Moreover, 
$0<\alpha<\frac{1+\alpha}{3}=\frac{2\|y_1+y_2\|}{3(\|y_1\|+\|y_2\|)}$.

Now, we have $\frac{y_1+y_2}{\|y_1+y_2\|}=x\in M_f$ for all $f\in J(x)$ and, by Lemma~\ref{lemma1}, approximate Birkhoff-James orthogonality is not right-additive at $x$. 

To justify essentialness of approximate smoothness in Theorem \ref{Cor1} take $x,z_1,z_2\in S_X$, where $x=(1,1)$, $z_1=
(1,-\alpha)$, $z_2=(-\alpha,1)$, $0<\alpha<\frac{1}{2}$ (see Figure 2(B)). Then
$x\perp_Bz_1$, $x\perp_Bz_2$ and clearly, $\frac{z_1+z_2}{\|z_1+z_2\|}=x\in M_f$ for all $f\in J(x)$. Thus Lemma~\ref{lemma1} implies that Birkhoff-James orthogonality is not approximately right-additive at $x$. 
\end{example}

\begin{figure} \label{fig1}
\centering
\begin{subfigure}{.50\textwidth}
\begin{tikzpicture}[scale=1]
\draw (1.2,1.2)  node {$x$};

\draw (0.3,1.3)  node {$y_2$};

\draw (1.3,0.2)  node {$y_1$};

\draw[thick] (-1,1)--(1,1) node[right]{};

\draw[thick] (-1,-1)--(1,-1) node[right]{};

\draw[thick] (1,1)--(1,-1) node[right]{};

\draw[thick] (-1,1)--(-1,-1) node[right]{};

\draw[thick, black,dashed] (-2,0)--(2,0) node[above]{};

\draw[thick,black,dashed] (0,-2)--(0,2) node[above]{};

\draw[thick,->] (0,0)--(1,1);
\draw[thick,->] (0,0)--(1,0.2);
\draw[thick,->] (0,0)--(0.2,1);
\end{tikzpicture}
\caption{}
\end{subfigure}% <- nötig
\begin{subfigure}{.50\textwidth}
\begin{tikzpicture}[scale=1]
\draw[thick, black,dashed] (-7,0)--(-3,0) node[above]{};

\draw[thick,black,dashed] (-5,-2)--(-5,2) node[above]{};
\draw (-3.8,1.2)  node {$x$};

\draw (-3.7,-.2)  node {$z_1$};

\draw[thick] (-6,1)--(-4,1) node[right]{};

\draw[thick] (-6,-1)--(-4,-1) node[right]{};

\draw[thick] (-4,1)--(-4,-1) node[right]{};

\draw[thick] (-6,1)--(-6,-1) node[right]{};

\draw (-5.3,1.2)  node {$z_2$};

\draw[thick,->] (-5,0)--(-4,1);
\draw[thick,->] (-5,0)--(-5.2,1);
\draw[thick,->] (-5,0)--(-4,-0.2);
\end{tikzpicture}
\caption{}
\end{subfigure}

\caption{Illustration to Example \ref{ex1}}
\end{figure}
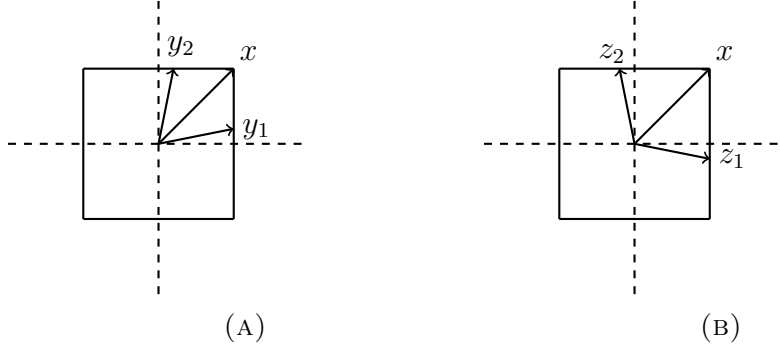

Finally in this section, we give a result showing that approximate smoothness of a vector $x$ guarantees the approximate additivity of the Birkhoff-James orthogonality on some restricted set of vectors.

\begin{theorem}
Let $X$ be a real normed linear space and let $x\in X\smallsetminus\{\theta\}$ be $\varepsilon$-smooth with $\varepsilon\in [0,2)$. Suppose that $y_1,y_2\in X$ are such that $\min\{\|y_1\|,\|y_2\|\}\leq\left\|\frac{y_1+y_2}{2}\right\|$. 
If $x\bot_B\, y_1$ and $x\bot_B\, y_2$, then 
$$
x\bot_B^{\varepsilon / 2}\,(y_1+y_2).
$$
\end{theorem}
\begin{proof}
Since $x\bot_B\, y_1$ and $x\bot_B\, y_2$, there exist $f,g\in J(x)$ such that $f(y_1)=g(y_2)=0$. Therefore, 
$$
f(y_2)=(f-g)(y_2)\ \ \mbox{and}\ \ g(y_1)=(g-f)(y_1)
$$
which, together with the assumed $\diam J(x)\leq\varepsilon$,	gives
$$
|f(y_2)|\leq\varepsilon\|y_2\|\ \ \mbox{and}\ \ |g(y_1)|\leq\varepsilon\|y_1\|.
$$
Suppose that $\|y_1\|\leq\|y_2\|$. It follows then from the assumption that $\|y_1\|\leq \left\|\frac{y_1+y_2}{2}\right\|$ and
$$
|g(y_1+y_2)|=|g(y_1)|\leq\varepsilon\|y_1\|\leq \frac{\varepsilon}{2}\|y_1+y_2\|.
$$   
This means that $x\bot_B^{\varepsilon / 2}\,(y_1+y_2)$.

Similarly, if $\|y_2\|\leq\|y_1\|$ we show that 
$|f(y_1+y_2)|\leq \frac{\varepsilon}{2}\|y_1+y_2\|$,
which also gives the assertion. 
\end{proof}

Note that the condition $\min\{\|y_1\|,\|y_2\|\}\leq\left\|\frac{y_1+y_2}{2}\right\|$ depends both on directions and norms of vectors $y_1,y_2$. It holds true however, regardless of directions, if $\|y_1\|\geq 3\|y_2\|$ or $\|y_2\|\geq 3\|y_1\|$.

\section{Polyhedral spaces}

Now, we consider a $2$-dimensional regular polyhedral Banach space $X$ with $2n$ extreme points. Regularity here means that all the edges of the unit sphere are of the same length with respect to the Euclidean metric and all the interior angles are of the same measure.
For such spaces we will calculate the values of $d(x)$ for each $x\in \Ext B_X$ and the value of $\mathcal{E}(X)$. 

\begin{example}\label{regular polygon}
Let $X$ be a $2$-dimensional regular polyhedral Banach space with $2n$ ($n\geq 2$) extreme points and let $x\in \Ext B_X$. Then,
$$
{d(x)}=\mathcal{E}(X)= \left\{
    \begin{array}{ll}
        2 \tan \frac{\pi}{2n} & \mbox{if~} n\mbox{ ~is~ even,}  \\
       2 \tan \frac{\pi}{2n} \sin\frac{(n-1)\pi}{2n} & \mbox{if}~n~\mbox{ is ~odd. } 
    \end{array}
\right.
$$ 
\end{example}

\begin{proof}
If $x\in S_X\smallsetminus \Ext B_X$ then clearly $d(x)=0$. Thus, to calculate $\mathcal{E}(X)$ it is sufficient to find $d(x)$ for all $x\in \Ext B_X$. Moreover,  regularity and symmetry of $S_X$ implies that $d(x)=d(y)$ for all $x,y\in \Ext B_X$.

Without loss of generality we may assume that 
$$
\Ext B_X=\{x_k:\ 1\leq k \leq 2n\},
$$
where $x_k=\left(\cos\frac{(k-1)\pi}{n},\sin\frac{(k-1)\pi}{n}\right)\ \mbox{for}\ 1\leq k\leq 2n
$
(see Figure 3). %\ref{fig3} - does not work

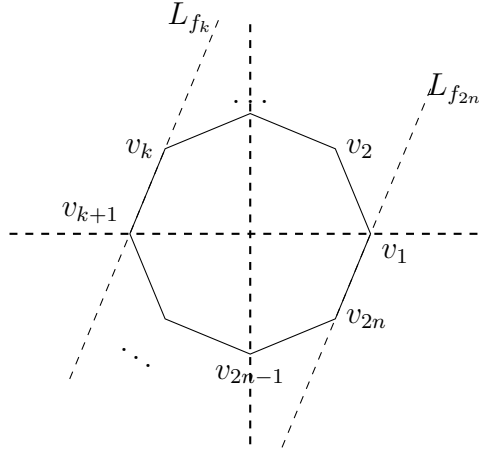
\begin{figure} \label{fig3}
\centering
\begin{tikzpicture}[scale=0.8]
\newdimen\R
\R=2cm
\draw (0:\R)
   \foreach \x in {45,90,135,180,225,270,315,360} {  -- (\x:\R) }
-- cycle (315:\R) node[right] {$v_{2n}$}
-- cycle (270:\R) node[below] {$v_{2n-1}$}
-- cycle (225:\R) node[below left] {$\ddots$} 
-- cycle (180:\R) node[above left] {$v_{k+1}$}
-- cycle (135:\R) node[left] {$v_k$}
-- cycle (90:\R)  node[above] {$\ldots$}  
-- cycle (45:\R)  node[right] {$v_2$} 
-- cycle (0:\R)   node[below right] {$v_1$};

\draw[dashed] (3,2.4141) -- (0.5,-3.62115);

\draw[dashed] (-3,-2.4141) -- (-0.5,3.62115);

\draw(3.4, 2.4141) node{$L_{f_{2n}}$};

\draw(-1, 3.6) node{$L_{f_{k}}$};

\draw[thick, black,dashed] (-4,0)--(4,0) node[above]{};

\draw[thick,black,dashed] (0,-3.5)--(0,3.5) node[above]{};

\end{tikzpicture}
\caption{Illustration to Example \ref{regular polygon}}
\end{figure}

Let $f_k$ be the unique support functional for the segment $\overline{x_kx_{k+1}}$ for $k=1,\ldots,2n-1$ and $\overline{x_{2n}x_{1}}$ for $k=2n$ (on Figure 3 the respective supporting lines are denoted by $L_{f_k}$ and $L_{f_{2n}}$). Some calculations, which will be omitted here, lead to the explicit formula for the value of $f_k$ at $(x,y)\in \mathbb{R}^2$:
$$
f_k((x,y))=\left(x \cos\frac{(2k-1)\pi}{2n} +y \sin\frac{(2k-1)\pi}{2n}\right)\sec \frac{\pi}{2n}
$$ 
(compare with a similar result in \cite[the proof of Theorem 3.1]{SRBB}).
Moreover, 
$$
J(x_k)= \left\{
    \begin{array}{ll}
        \overline{f_{k-1}f_k} & \mbox{if~} 1< k\leq 2n,  \\
       \overline{f_{2n}f_{1}} & \mbox{if } k=1
    \end{array}
\right.
$$
and thus 
$$
{d(x_k)}= \left\{
    \begin{array}{ll}
        \|f_{k-1}-f_k\| & \mbox{if~} 1< k\leq 2n,  \\
       \|f_{2n}-f_{1}\| & \mbox{if } k=1.
    \end{array}
\right.
$$ 
For $1< k\leq 2n$, we have
\begin{align*}
    (f_k-f_{k-1})((x,y))=x\left(\cos\frac{(2k-1)\pi}{n}-\cos\frac{(2k-3)\pi}{n}\right)\sec\frac{\pi}{2n}\\
    +y\left(\sin\frac{(2k-1)\pi}{n}-\sin\frac{(2k-3)\pi}{n}\right)\sec\frac{\pi}{2n}\\
    =\left\{-2x\sin\frac{(k-1)\pi}{n}\sin\frac{\pi}{2n}+2y\cos
    \frac{(k-1)\pi}{n}\cos\frac{\pi}{2n}\right\}\sec\frac{\pi}{2n}.
    \end{align*}
    To calculate the norm of $f_k-f_{k-1}$ (supremum over the unit sphere) we use the Krein-Milman theorem and restrict ourselves to extremal points. Thus we have for $1<k\leq 2n$
\begin{equation}\label{eq1}
\begin{array}{l}
\|f_k-f_{k-1}\|=\\
\max_{(x,y)\in \Ext B_X}\left|-2x\sin\frac{(k-1)\pi}{n}\sin\frac{\pi}{2n}+2y\cos\frac{(k-1)\pi}{n}\cos\frac{\pi}{2n}\right|\sec\frac{\pi}{2n}.
\end{array}    
\end{equation}
If $n$ is even then by taking $k=\frac{n}{2}+1$ in $(\ref{eq1})$, we get
$$
\|f_{\frac{n}{2}+1}-f_{\frac{n}{2}}\|=\max_{(x,y)\in \Ext B_X}2|x|\sin\frac{\pi}{2n}\sec\frac{\pi}{2n}.
$$
For $(x,y)\in \Ext B_X$, $|x|\leq 1$. Thus, $\|f_{\frac{n}{2}+1}-f_{\frac{n}{2}}\|=2 \tan \frac{\pi}{2n}$.

If $n$ is odd, then by taking $k=n+1$ in \eqref{eq1}, we get
\begin{equation}\label{eq2}
\|f_{n+1}-f_{n}\|=\max_{(x,y)\in \Ext B_X}2|y|\sin\frac{\pi}{2n}\sec\frac{\pi}{2n}.
\end{equation}
For $(x,y)\in \Ext B_X$, $|y|\leq \sin\frac{(n-1)\pi}{2n}$. Thus,
\begin{equation}\label{eq3}
\|f_{n+1}-f_{n}\|=2\sin\frac{(n-1)\pi}{2n}\sin\frac{\pi}{2n}\sec\frac{\pi}{2n}=2 \tan\frac{\pi}{2n}\sin\frac{(n-1)\pi}{2n}.
\end{equation}
Now, using \eqref{eq2}, \eqref{eq3}, symmetry and regularity of $S_X$, we get
$$
\mathcal{E}(X)={d(x_k)}= \left\{
    \begin{array}{ll}
        2 \tan \frac{\pi}{2n} & \mbox{if}~n~\mbox{is~ even},  \\
       2 \tan \frac{\pi}{2n} \sin\frac{(n-1)\pi}{2n} & \mbox{if}~n~\mbox{is ~odd. } 
    \end{array}
\right.
$$

\end{proof}
As an application of the above, we obtain the following result on approximate smoothness of a $2$-dimensional regular polyhedral Banach space $X$ with $2n$ extreme points.

\begin{proposition}
Let $X$ be a $2$-dimensional regular polyhedral Banach space with $2n$ extreme points ($n\geq 2$). Then the following hold true.
\begin{itemize}
    \item[(i)] $X$ is $\varepsilon$-smooth for $\varepsilon\in[2 \tan \frac{\pi}{2n},2)$, if $n$ is even.
    \item[(ii)] $X$ is $\varepsilon$-smooth for $\varepsilon\in[2  \tan \frac{\pi}{2n} \sin\frac{(n-1)\pi}{2n},2)$, if $n\geq 3$ is odd.
\end{itemize}
\end{proposition}

In the next result we give formulas for $ d(x)$, $x\in S_X$ and $\mathcal{E}(X)$ for a finite-dimensional polyhedral Banach space $X$. Note that in case of $X$ being a finite-dimensional polyhyderal Banach space, its dual $X^*$ is also a polyhedral Banach space with finitely many extreme points in $B_{X^*}$.

\begin{theorem}
Let $X$ be a finite-dimensional polyhedral Banach space and let $x\in S_X$. Then 
\begin{itemize}
    \item [(i)]
$d(x)=\max\{\|f_i-f_j\|:\ f_i,f_j\in \Ext B_{X^*}\ \mbox{such ~that}\ f_i,f_j\in J(x)\},
$

\item[(ii)] $
\mathcal{E}(X)=\max\{\|f_i-f_j\|:\ f_i,f_j\in \Ext B_{X^*}\ \mbox{such that}\ M_{f_i}^{+}\cap M_{f_j}^{+}\not=\emptyset\}.
$
\end{itemize}
\end{theorem}

\begin{proof}
{(i)}   Let $x\in S_X$ and $F_{i_1},F_{i_2},\ldots,F_{i_k}$ be the facets of $S_X$ which contain $x$. Let $f_{i_1},f_{i_2},\ldots, f_{i_k}\in \Ext B_{X^*}$ be the unique supporting functionals for the facets $F_{i_1},F_{i_2},\ldots,F_{i_k}$, respectively.  Let $f,g\in J(x)$. Then $f=\sum_{j=1}^{k}\alpha_j f_{i_j}$, $g=\sum_{j=1}^{k}\beta_j f_{i_j}$, where $0\leq \alpha_j,\beta_j\leq 1$  and $\sum_{j=1}^k\alpha_j=1$, $\sum_{j=1}^k\beta_j=1$. Now,
\begin{eqnarray*}
    \|f-g\|&=&\left\|\sum_{j=1}^{k}\alpha_j f_{i_j}-g\right\|
           =\left\|\sum_{j=1}^{k}\alpha_j f_{i_j}-\sum_{j=1}^{k}\alpha_j g\right\|\\
           &=&\left\|\sum_{j=1}^{k}\alpha_j(f_{i_j}-g)\right\|
           \leq \sum_{j=1}^{k}\alpha_j\|f_{i_j}-g\|.
\end{eqnarray*}
Similar arguments show that $$
\|f_{i_j}-g\|\leq \sum_{\ell=1}^{k}\beta_\ell\|f_{i_j}-f_{i_\ell}\|\leq \max\limits_{1\leq \ell\leq  k}\|f_{i_j}-f_{i_\ell}\|
$$ 
for all $1\leq j\leq k$. Thus, $\|f-g\|\leq \max\limits_{1\leq j,\ell\leq  k}\|f_{i_j}-f_{i_\ell}\|$ and this proves the result.

(ii) Observe that if $y\in F$ for some facet $F$ and $f,g\in J(y)$ then $f,g\in J(x)$ for some $x\in \Ext B_X\cap F$. Thus to calculate $\mathcal{E}(X)$ it is sufficient to consider $x\in \Ext B_X$. Now, (ii) follows from (i) by observing that if $f,g\in J(x)$ for some $x\in \Ext B_X\cap F$ and facet $F$ then $\ M_{f}^{+}\cap M_{g}^{+}\not=\emptyset$. 
\end{proof}

\section{Direct sums}
Given normed linear spaces $X,Y$ we study the space $Z=X\oplus_p Y$, $1\leq p\leq \infty$, and approximate smoothness of its elements.
We start with a description of the set of supporting functionals.

\begin{proposition}\label{dirsum}
Let $X,Y$ be normed linear spaces and let $Z=X\oplus_p Y$ with $1\leq p\leq \infty$. 
\begin{enumerate}
\item\label{p}
If $1<p<\infty$ and $q$ is conjugated to $p$, then for any $x\in X\smallsetminus\{\theta\}$ and $y\in Y\smallsetminus\{\theta\}$ we have:
\begin{enumerate}
    \item\label{pa} \begin{equation}\label{dirsumeq1}
\begin{array}{l}
J((x,y))=\\
\left\{
\left(
\frac{\|x\|^{p-1}f}{(\|x\|^p+\|y\|^p)^{\frac{1}{q}}},\frac{\|y\|^{p-1}g}{(\|x\|^p+\|y\|^p)^{\frac{1}{q}}}
\right)\in S_{X^*\oplus_q Y^*}:\ 
f\in J(x),\
g\in J(y)
\right\},
\end{array}
\end{equation}
\item \label{p(b)}$J((x,\theta))=\{
(f,\theta)\in S_{X^*\oplus_q Y^*}:\ 
f\in J(x)
\}$,

\item\label{p(c)} $J((\theta,y))=\{
(\theta,g)\in S_{X^*\oplus_q Y^*}:\ 
g\in J(y)
\}$.
\end{enumerate}

\item\label{p1}
If $p=1$, $x\in X\smallsetminus\{\theta\}$, $y\in Y\smallsetminus\{\theta\}$, then
\begin{enumerate}

\item\label{p1c}
$J((x,y))=\{(f,g)\in S_{X^*\oplus_{\infty} Y^*}:\ f\in J(x),\ g\in J(y)\}$,

\item\label{p1a}

$J((x,\theta))=\{(f,g)\in S_{X^*\oplus_{\infty} Y^*}:\ f\in J(x),\ g\in B_{Y^*}\}, 
$
\item\label{p1b}
$
J((\theta,y))=\{(f,g)\in S_{X^*\oplus_{\infty} Y^*}:\ f\in B_{X^*},\ g\in J(y)\}.
$

\end{enumerate}
\item\label{pinfty}
Let $p=\infty$ and $(x,y)\in Z\smallsetminus\{\theta\}$.
\begin{enumerate}
\item\label{pinftya}
If $\|x\|>\|y\|$, then 
$$
J((x,y))=\{(f,\theta)\in S_{X^*\oplus_{1} Y^*}:\ f\in J(x)\}.
$$
\item\label{pinftyb}
If $\|x\|<\|y\|$, then
$$
J((x,y))=\{(\theta,g)\in S_{X^*\oplus_{1} Y^*}:\ g\in J(y)\}.
$$
\item\label{pinftyc}
If $\|x\|=\|y\|$, then
$$
\{(\alpha f, (1-\alpha)g): f\in J(x), g\in J(y), 0\leq\alpha\leq 1\}\subseteq J((x,y)).
$$
\end{enumerate}
\end{enumerate}
\end{proposition}
\begin{proof}

Ad \eqref{pa}. Let $F\in J((x,y))$, where $F=(f,g)\in S_{X^*\oplus_q Y^*}$ . Then using H\"older's inequality, we get
\begin{eqnarray*}
   ( \|x\|^p+\|y\|^p)^{\frac{1}{p}}&=&\|(x,y)\|=F((x,y))=f(x)+g(y)\\
   &\leq& \|f\|\|x\|+\|g\|\|y\|\\
   & \leq& (\|f\|^q+\|g\|^q)^{\frac{1}{q}} (\|x\|^p+\|y\|^p)^{\frac{1}{p}}\\
   &=&(\|x\|^p+\|y\|^p)^{\frac{1}{p}}.
\end{eqnarray*}

This shows that equality holds in H\"older's inequality and thus
\begin{equation}\label{one-two}
\|f\|=\frac{\|x\|^{p-1}}{(\|x\|^p+\|y\|^p)^{\frac{1}{q}}},~
\|g\|=\frac{\|y\|^{p-1}}{(\|x\|^p+\|y\|^p)^{\frac{1}{q}}}.
\end{equation}

Also, we get $f(x)=\|f\|\|x\|$, $g(y)=\|g\|\|y\|$ which implies that $\tilde{f}=\frac{f}{\|f\|}\in J(x)$ and $\tilde{g}=\frac{g}{\|g\|}\in J(y)$. Combining this with \eqref{one-two} we get 
$$
  F=(f,g)=\left(\frac{\|x\|^{p-1}\tilde{f}}{(\|x\|^p+\|y\|^p)^\frac{1}{q}}, \frac{\|y\|^{p-1}\tilde{g}}{(\|x\|^p+\|y\|^p)^\frac{1}{q}}\right)
$$
which shows $\subseteq$ in \eqref{dirsumeq1}. 
Now, consider 
$$
F=\left(\frac{\|x\|^{p-1}f}{(\|x\|^p+\|y\|^p)^\frac{1}{q}}, \frac{\|y\|^{p-1}g}{(\|x\|^p+\|y\|^p)^\frac{1}{q}}\right)\in S_{X^*\oplus_q Y^*},
$$ 
where $f\in J(x)$ and $g\in J(y)$. Then (we use that $(p-1)q=p$)
$$
F((x,y))=\frac{\|x\|^p+\|y\|^p}{(\|x\|^p+\|y\|^p)^\frac{1}{q}}=\|(x,y)\|.
$$
Thus $F\in J((x,y))$ and this proves $\supseteq$ in \eqref{dirsumeq1} and finishes the proof of part \eqref{pa}.

Parts \eqref{p(b)} and \eqref{p(c)} follow using the similar reasoning.

%\bigskip

Ad \eqref{p1}. 
For the proof of \eqref{p1c} let $x\in X\setminus\{\theta\}$, $y\in Y\setminus\{\theta\}$, $f\in J(x)$ and $y\in J(y)$. Then $(f,g)\in S_{ X^*\oplus_\infty Y^*}$ and $(f,g)((x,y))=f(x)+g(y)=\|x\|+\|y\|=\|(x,y)\|$. Thus $(f,g)\in J((x,y))$. 

For the reverse, let $(f,g)\in J((x,y))$. Then $\|f\|\leq 1$, $\|g\|\leq 1$ and
\begin{eqnarray*}
 \|x\|+\|y\|&=&\|(x,y)\|=(f,g)((x,y))=f(x)+g(y)\\
 &\leq& \|f\|\|x\|+\|g\|\|y\|\leq \|x\|+\|y\|.
\end{eqnarray*}
This shows that $\|f\|=1$, $\|g\|=1$,  $f(x)=\|x\|$ and $g(y)=\|y\|$. Thus $f\in J(x)$ and $g\in J(y)$.

To prove \eqref{p1a} let $x\in X\setminus \{\theta\}$, $f\in J(x)$ and $g\in B_{Y^*}$. Then $(f,g)\in S_{ X^*\oplus_\infty Y^*}$ and $(f,g)((x,\theta))=f(x)=\|x\|=\|(x,\theta)\|$. Thus $(f,g)\in J((x,\theta))$. The reverse inclusion is clear.

The proof of \eqref{p1b} is analogous.

%\bigskip

Ad \eqref{pinfty}. For the proof of \eqref{pinftya} let $x\in X\setminus \{\theta\}$, $y\in Y$ be such that $\|x\|>\|y\|$. Let $f\in J(x)$. Then clearly $(f,\theta)\in S_{X^*\oplus_1Y^*}$, $(f,\theta)((x,y))=f(x)=\|x\|=\|(x,y)\|$ and thus $(f,\theta)\in J((x,y))$.

Now, let $(f,g)\in J((x,y))$, where $g\in S_{Y^*}$. If $f=\theta$,  then 
$$
(f,g)((x,y))=g(y)\leq \|y\|<\|(x,y)\|.
$$ 
Thus $f\not=\theta$. Since $(f,g)\in S_{ X^*\oplus_1 Y^*}$, $f\not=\theta$, $g\not=\theta$, this gives $\|f\|<1$ and $\|g\|<1$.  Also, $\|f\|+\|g\|=1$ implies $\|g\|=1-\|f\|$. Thus 
\begin{eqnarray*}
\|x\|&=&\|(x,y)\|=(f,g)((x,y))=f(x)+g(y)\\
&\leq&\|f\|\|x\|+\|g\|\|y\|<\|f\|\|x\|+(1-\|f\|)\|x\|=\|x\|.
\end{eqnarray*} 
The obtained contradiction proves the required form of $J((x,y))$.

The proof of \eqref{pinftyb} is analogous.

Finally, we prove \eqref{pinftyc}. 
Let $x\in X\setminus \{\theta\}$, $y\in Y\setminus \{\theta\}$ be such that $\|x\|=\|y\|$. Let $f\in J(x)$, $g\in J(y)$ and $0\leq \alpha\leq 1$. Then $(\alpha f, (1-\alpha)g)\in S_{X^*\oplus_1Y^*}$ and $(\alpha f, (1-\alpha)g)((x,y))=\alpha f(x)+(1-\alpha)g(y)=\alpha\|x\|+(1-\alpha)\|y\|=\|x\|=\|(x,y)\|$. This proves the result.
\end{proof}

\begin{corollary}\label{dirsumcor}
Let $X,Y$ be normed linear spaces and let $Z=X\oplus_p Y$ with $1\leq p\leq \infty$.
\begin{enumerate}
\item\label{corp}
If $1<p<\infty$ and $q$ is conjugated to $p$, then for any $x\in X\smallsetminus\{\theta\}$ and $y\in Y\smallsetminus\{\theta\}$ we have:
\begin{enumerate}
    \item\label{corpa} 
\begin{equation}\label{formuladxy}
d((x,y))=\left(
\frac{\|x\|^{p}}{\|x\|^p+\|y\|^p}d(x)^{q}+\frac{\|y\|^{p}}{\|x\|^p+\|y\|^p}d(y)^{q}
\right)^{\frac{1}{q}},
\end{equation}
\item \label{corp(b)} $d((x,\theta))=d(x)$, 
\item\label{corp(c)} $d((\theta,y))=d(y)$.
\end{enumerate}

\item\label{corp1}
If $p=1$, $x\in X\smallsetminus\{\theta\}$, $y\in Y\smallsetminus\{\theta\}$, then
\begin{enumerate}
\item\label{corp1c}
$d((x,y))=\max\{d(x),d(y)\}$,
\item\label{corp1a}
$d((x,\theta))=2$,
\item\label{corp1b}
$J((\theta,y))=2$.
\end{enumerate}

\item\label{corpinfty}
Let $p=\infty$, $(x,y)\in Z\smallsetminus\{\theta\}$.
\begin{enumerate}
\item\label{corpinftya}
If $\|x\|>\|y\|$, then 
$d((x,y))=d(x)$.
\item\label{corpinftyb}
If $\|x\|<\|y\|$, then
$d((x,y))=d(y)$.
\item\label{corpinftyc}
If $\|x\|=\|y\|$, then
$d((x,y))=2$.
\end{enumerate}
\end{enumerate}
\end{corollary}

\begin{proof}
We will prove \eqref{corpa}. 
Let $F,G\in J((x,y))$. On account of \eqref{dirsumeq1}, there exist $f_x,g_x\in J(x)$, $f_y,g_y\in J(y)$ such that $$
F=\left(\frac{\|x\|^{p-1}f_x}{(\|x\|^p+\|y\|^p)^\frac{1}{q}}, \frac{\|y\|^{p-1}f_y}{(\|x\|^p+\|y\|^p)^\frac{1}{q}}\right)
$$
and
$$
G=\left(\frac{\|x\|^{p-1}g_x}{(\|x\|^p+\|y\|^p)^\frac{1}{q}}, \frac{\|y\|^{p-1}g_y}{(\|x\|^p+\|y\|^p)^\frac{1}{q}}\right).
$$ 
This gives (using $(p-1)q=p$)

%\begin{equation}\label{cor1}
$$
\|F-G\|^q= \frac{\|x\|^{p}}{\|x\|^p+\|y\|^p }\|f_x-g_x\|^q+\frac{\|y\|^{p}}{\|x\|^p+\|y\|^p }\|f_y-g_y\|^q
$$
%\end{equation}
and thus,
$$
d((x,y))^q \leq \frac{\|x\|^{p}}{\|x\|^p+\|y\|^p }d(x)^q+\frac{\|y\|^{p}}{\|x\|^p+\|y\|^p }d(y)^q.
$$
To prove the reverse inequality fix arbitrarily $\delta>0$ and choose $f_x^{\delta},g_x^{\delta}\in J(x)$ such that $\|f_x^{\delta}-g_x^{\delta}\|^q>d(x)^{q}-\delta$. Analogously, let $f_y^{\delta},g_y^{\delta}\in J(y)$ be chosen such that $\|f_y^{\delta}-g_y^{\delta}\|^q>d(y)^{q}-\delta$. Define
$$
F^{\delta}:=\left(\frac{\|x\|^{p-1}f_x^{\delta}}{(\|x\|^p+\|y\|^p)^\frac{1}{q}}, \frac{\|y\|^{p-1}f_y^{\delta}}{(\|x\|^p+\|y\|^p)^\frac{1}{q}}\right)
$$
and
$$
G^{\delta}:=\left(\frac{\|x\|^{p-1}g_x^{\delta}}{(\|x\|^p+\|y\|^p)^\frac{1}{q}}, \frac{\|y\|^{p-1}g_y^{\delta}}{(\|x\|^p+\|y\|^p)^\frac{1}{q}}\right).
$$ 
By \eqref{dirsumeq1}, $F^{\delta},G^{\delta}\in J((x,y))$ whence
\begin{eqnarray*}
d((x,y))^q&\geq&\|F^{\delta}-G^{\delta}\|^q\\
%&=&\frac{\|x\|^{p}}{\|x\|^p+\|y\|^p }\|f_x^{\delta}-g_x^{\delta}\|^q+\frac{\|y\|^{p}}{\|x\|%^p+\|y\|^p }\|f_y^{\delta}-g_y^{\delta}\|^q\\
&>&\frac{\|x\|^{p}}{\|x\|^p+\|y\|^p }(d(x)^q-\delta)+\frac{\|y\|^{p}}{\|x\|^p+\|y\|^p}(d(y)^q-\delta)
\end{eqnarray*}
and since $\delta>0$ was arbitrary, we get
$$
d((x,y))^q\geq \frac{\|x\|^{p}}{\|x\|^p+\|y\|^p }d(x)^q+\frac{\|y\|^{p}}{\|x\|^p+\|y\|^p}d(y)^q.
$$
The proofs of other cases are similar or obvious. 
\end{proof}

Now, as a straightforward consequence of Corollary \ref{dirsumcor}, we characterize approximate smoothness of the direct sum. 
%in terms of approximate smoothness of its components. 

\begin{theorem}
Let $X$, $Y$ be normed linear spaces and $Z=X\oplus_p Y$, $1<p<\infty$. Then the following statements hold true:
\begin{itemize}
    \item [(i)] Let $x\in X\smallsetminus\{\theta\}$, $y\in Y\smallsetminus\{\theta\}$.

\begin{itemize}

\item[(a)]    
If $x$ is $\varepsilon_x$-smooth in $X$ and $y$ is $\varepsilon_y$-smooth in $Y$ for $\varepsilon_x,\varepsilon_y\in [0,2)$, then $(x,y)$ is $\varepsilon$-smooth in $Z$ with
$$
\varepsilon:=\left(
\frac{\|x\|^{p}}{\|x\|^p+\|y\|^p}\varepsilon_x^{q}+\frac{\|y\|^{p}}{\|x\|^p+\|y\|^p}\varepsilon_y^{q}\right)^{\frac{1}{q}}.
$$
\item[(b)]    
If $x$ is $\varepsilon$-smooth in $X$ and $y$ is $\varepsilon$-smooth in $Y$ for $\varepsilon\in [0,2)$, then $(x,y)$ is $\varepsilon$-smooth in $Z$.
\item[(c)]    
If $(x,y)$ is $\varepsilon$-smooth in $Z$ for $\varepsilon\in [0,2)$, then either $x$ is $\varepsilon$-smooth in $X$ or $y$ is $\varepsilon$-smooth in $Y$.
\item[(d)]
$(x,y)$ is approximately smooth if and only if either $x$ or $y$ is approximately smooth. 
\end{itemize}

\item [(ii)] Let $x\in X\smallsetminus\{\theta\}$. Then $(x,\theta)\in Z$ is $\varepsilon$-smooth for $\varepsilon\in[0,2)$ if and only if $x$ is $\varepsilon$-smooth in $X$.

\item [(iii)] Let $y\in Y\smallsetminus\{\theta\}$. Then $(\theta,y)\in Z$ is  $\varepsilon$-smooth for $\varepsilon\in[0,2)$ if and only if  $y$ is $\varepsilon$-smooth in $Y$.
\end{itemize}

\end{theorem}

\begin{proof}
For the proof of (i) we need to apply Corollary \ref{dirsumcor} \eqref{corpa} and the formula \eqref{formuladxy}. In particular, for (id) we observe that $d(x,y)<2$ if and only  if $d(x)<2$ or $d(y)<2$. 
The statements (ii) and (iii) immediately follow from Corollary \ref{dirsumcor} \eqref{corp(b)} and \eqref{corp(c)}, respectively. 
\end{proof}

\begin{theorem}
Let $X$, $Y$ be normed linear spaces and $Z=X\oplus_1 Y$. Then the following hold true:
\begin{itemize}
    
\item[(i)] 
   
\begin{itemize}
\item[(a)]    
If $x$ is $\varepsilon_x$-smooth in $X$ and $y$ is $\varepsilon_y$-smooth in $Y$ for $\varepsilon_x,\varepsilon_y\in [0,2)$, then $(x,y)$ is $\varepsilon$-smooth in $Z$ with
$\varepsilon:=\max\{\varepsilon_x,\varepsilon_y\}$.
\item[(b)]    
If $x$ is $\varepsilon$-smooth in $X$ and $y$ is $\varepsilon$-smooth in $Y$ for $\varepsilon\in [0,2)$, then $(x,y)$ is $\varepsilon$-smooth in $Z$.
\item[(c)]    
If $(x,y)$ is $\varepsilon$-smooth in $Z$ for $\varepsilon\in [0,2)$, then both $x$ and $y$ are $\varepsilon$-smooth in $X$ and $Y$, respectively.
\item[(d)]
$(x,y)$ is approximately smooth if and only if either $x$ or $y$ is approximately smooth. 
\end{itemize}

\item [(ii)] If $x\in X\smallsetminus\{\theta\}$ then $(x,\theta)\in Z$  cannot be approximately smooth.
\item [(iii)] If $y\in Y\smallsetminus\{\theta\}$ then $(\theta,y)\in Z$  cannot be approximately smooth.   
   
\end{itemize}
\end{theorem}

\begin{proof}

The proof of (i) relies on Corollary \ref{dirsumcor} \eqref{corp1c}.
Statements (ii) and (iii) follow immediately from \eqref{corp1a} and \eqref{corp1b}, respectively.
\end{proof}

\begin{theorem}
Let $X$, $Y$ be normed linear spaces, $Z=X\oplus_\infty Y$ and $z=(x,y)\in Z\smallsetminus\{\theta\}$. Then the following statements hold true:
\begin{itemize}
    \item [(i)] Let $\|x\|>\|y\|$.  Then  $z$ is $\varepsilon$-smooth in $Z$ for $\varepsilon\in[0,2)$ if and only if  $x$ is $\varepsilon$-smooth in $X$. 
    \item [(ii)] Let $\|x\|<\|y\|$. Then $z$ is $\varepsilon$-smooth in $Z$ for $\varepsilon\in[0,2)$ if and only if $y$ is $\varepsilon$-smooth in $Y$.
   \item  [(iii)] If $\|x\|=\|y\|$, then $z$ cannot be approximately smooth.
\end{itemize}
\end{theorem}

\begin{proof}
Statements (i) and (ii) follows from Corollary \ref{dirsumcor} \eqref{corpinftya} and \eqref{corpinftyb}, whereas (iii) is a consequence of Corollary \ref{dirsumcor} \eqref{corpinftyc}.
\end{proof}

The following final example is an application of the above theorem and the fact that any $t\in \mathbb{R}\smallsetminus\{0\}$ is a smooth point of $\mathbb{R}$.

\begin{example}
Let $Z=X\oplus_\infty \mathbb{R}$ be a 3-dimensional polyhedral Banach space whose unit ball is a right prism with regular polyhedron $P$ as its base. Then the following hold true for $z=(x,t)\in Z\smallsetminus\{\theta\}$:
\begin{itemize}
    \item [(i)] Let $\|x\|>|t|$.  Then  $z$ is $\varepsilon$-smooth in $Z$ for $\varepsilon\in[0,2)$ if and only if  $x$ is $\varepsilon$-smooth in $X$.
    \item [(ii)] Let $\|x\|<|t|$. Then $z$ is smooth.
   \item  [(iii)] If $\|x\|=|t|$, then $z$ cannot be approximately smooth.
\end{itemize}
\end{example}
\section{Acknowledgments}
The research of Divya Khurana is sponsored by
Dr. D. S. Kothari Postdoctoral Fellowship under the mentorship of Professor Gadadhar Misra. She would also like to thank Professor Gideon Schechtman for discussion on direct sum of normed linear spaces. The research of Dr. Debmalya Sain is sponsored by DST-SERB under the mentorship of Professor Apoorva Khare. Dr. Sain is elated to acknowledge the monumental positive role played by the Department of Mathematics, Indian Institute of Science, for providing him with a rich mathematical setting.

%The authors state that there is no conflict of interest.

%Data sharing not applicable to this article as no datasets were generated or analysed %during the current study.

\bibliographystyle{amsplain}

\end{document}